\documentclass[a4paper, 10pt]{amsart}
\usepackage{amscd, amsthm, amsmath, amssymb, amsfonts, mathrsfs, latexsym, epsfig}
\usepackage{graphicx}
\usepackage[cp1250]{inputenc}
\usepackage{geometry}
\newgeometry{tmargin=2cm, bmargin=2cm, lmargin=2cm, rmargin=2cm}

\renewcommand{\tilde}{\widetilde}
\newtheorem{df}{Definition}[section]
\newtheorem{lem}[df]{Lemma}
\newtheorem{tw}[df]{Theorem}
\newtheorem{pp}[df]{Proposition}
\newtheorem{col}[df]{Corollary}

\newtheorem{rem}[df]{Remark}
\newtheorem{ex}[df]{Example}
\newtheorem{qu}[df]{Question}
\newtheorem*{th1}{Theorem}
\author{Piotr Pokora}

\title{Rank function equations and their solution sets}

\begin{document}

\begin{abstract}
We examine so-called rank function equations and their solutions consisting of non-nilpotent matrices. Secondly, we present some geometrical properties of the set of solutions to certain rank function equations in the nilpotent case. The main results are Theorem 3.2 and Theorem 4.5. 
\end{abstract}

\maketitle
\section{Introduction to the subject}

Conjugacy classes of matrices were studied by many authors in various settings, especially in the context of algebraic geometry. There are many deep geometrical results describing some properties of conjugacy classes, in particular we are interested here in some analogous of theorems due to Gerstenhaber -- see for example \cite{Ger}. These theorems can be formulated using the notion of rank functions of matrices $r_{A}(m) := {\rm rk}(A^{m})$ as introduced by Eisenbud and Saltman \cite{Eis}. 
Gerstenhaber investigated closures of conjugacy classes. For nilpotent matrices he obtained the following well-known theorem \cite{Ger}.
\begin{th1}
Let $A,B$ be $n \times n$ nilpotent matrices with entries from an arbitrary algebraically closed field and let $\mathcal{O}(B)$ denotes the conjugacy class of $B$ under the standard action of the general linear group by conjugation, see $(\bullet)$. Then
$$A \in \overline{\mathcal{O}(B)} \text{ if and only if } r_{A}(m) \leq r_{B}(m) \text{ for all } m \geq 0.$$
\end{th1}

In this note we extand this result to the Cartesian product set up, i.e. we obtain the following result, see also Theorem 4.5.
\begin{th1} Let $A_{1}, ..., A_{k}, B_{1}, ..., B_{k}$ be $n \times n$ nilpotent matrices matrices with entries from $\mathbb{C}$. Then 
$$ (A_{1}, ..., A_{k}) \in \overline{\mathcal{O}(B_{1}) \times ... \times \mathcal{O}(B_{k})} \text{ if and only if } r_{A_{i}}(m) \leq r_{B_{i}}(m) \text{ for all } m \geq 0 \text{ and } i \in \{1, ..., k\}.$$
\end{th1}
Keeping the notion of the above theorem and taking $B = B_{1} \oplus ... \oplus B_{k}$ one can be interested in solutions to the following inequality 
$$ (\bullet \bullet) \,\,\,\,\,\,\, r_{A_{1}}(m) + ... + r_{A_{k}}(m) \leq r_{B}(m).$$
Finding all solutions to $(\bullet \bullet)$ without additional assumptions on the matrices $A_{1}, ..., A_{k}, B$ seems to be a quite complicated task. However if we replace the inequality by the equality in $(\bullet \bullet)$, then more can be said. In fact we can generalize the equation 
$$r_{A_{1}}(m) + ... + r_{A_{k}}(m) = r_{B}(m)$$
to the following problem 
$$f(r_{A_{1}}(m)) + ... + f(r_{A_{k}}(m)) = g(r_{B}(m))$$
with arbitrary functions $f,g : \mathbb{N} \rightarrow \mathbb{N}$. This is called \emph{rank function equation}, see Definition \ref{df:1}. 
A somewhat technical statement concerning solutions of certain rank function equations is formulated in Theorem \ref{st1} and Theorem \ref{st111}, which are the other main results of this note.
\section{Preliminaries}

Throughout this paper we assume that $\mathbb{F}$ is an arbitrary field of characteristic zero.  We denote by $\mathbb{N}_{0}$ the set of all positive integers and by $\mathbb{N}$ the set of all non-negative integers. For $n \in \mathbb{N}_{0}$ we define $M_{n \times n}(\mathbb{F})$ to be the ring of all $n \times n$ matrices whose entries are elements of the field $\mathbb{F}$. This ring has a natural structure of a $n^2$-dimensional $\mathbb{F}$-vector space. We denote the zero matrix by $O_{n}$. The set of all nonsingular $n \times n$ matrices over $\mathbb{F}$ will be denoted by $GL(n, \mathbb{F})$. Finally, let $\mathfrak{N}_{n}$ be the set of all nilpotent $ n \times n$ matrices over $\mathbb{F}$. The group $GL(n, \mathbb{F})$ acts on $M_{n \times n}(\mathbb{F})$ and $\mathfrak{N}_{n}$ by conjugation. The conjugacy class $\mathcal{O}(A)$ of a matrix $A \in M_{n \times n}(\mathbb{F})$ is defined by 
$$ (\bullet) \,\,\,\,\,\,\, \mathcal{O}(A) = \{ U^{-1}AU : U \in GL(n, \mathbb{F}) \}.$$
By $\overline{\mathcal{O}(A)}$ we denote the Zariski closure of the conjugacy class of a matrix $A$ in $M_{n \times n}(\mathbb{F})$. 

We refer to \cite{Gan} for matrix theory and to \cite{Sh} for algebraic geometry. 

\begin{df}
The matrix 
$$N_{k} = \left( \begin{array}{cccccc} 
0 & 1 & 0 & ... & 0 & 0 \\
0 & 0 & 1 & ... & 0 & 0 \\
0 & 0 & 0          & ... & 0 & 0 \\
\vdots & \vdots & \vdots & \vdots & \vdots & \vdots \\
0 & 0 & 0 & ... & 0 & 1 \\
0 & 0 & 0 & ... & 0 & 0
\end{array} \right) \in M_{k \times k} (\mathbb{F})$$
is called the \emph{Jordan nilpotent block} of size $k$.
\end{df}
Such matrices are building blocks of all nilpotent matrices as the following classical results shows.
\begin{pp}
\label{postacjordana}
Let $A \in M_{n \times n} (\mathbb{F})$ be a nilpotent matrix. Then there exist $U \in GL(n, \mathbb{F})$, $\ell \in \mathbb{N} \setminus \{0\}$, and a weakly decreasing sequence $(k_{1} ,..., k_{\ell})$ of positive integers such that $U^{-1} A U = N_{k_{1}} \oplus ... \oplus N_{k_{\ell}}$. Moreover, $\ell$ and $(k_{1}, ..., k_{\ell})$ are uniquely determined by the matrix $A$.
\end{pp}

\begin{df}
The matrix $N_{k_{1}} \oplus ... \oplus N_{k_{\ell}}$ is referred to as the \emph{Jordan canonical form} of $A$ and the related sequence ${\rm jp}(A) =(k_{1}, ..., k_{\ell})$ is called the \emph{Jordan partition}.
\end{df}
For a nilpotent matrix $A$ with Jordan canonical form $N_{k_{1}} \oplus ... \oplus N_{k_{\ell}}$, we denote by $\tilde{A}$ the direct sum of all non-trivial nilpotent blocks, i.e. those with $k_{j} \geq 2$.

Now we recall some facts related to rank functions and rank function equations. For more details we refer to \cite{PS} and \cite{Sk1}.
\begin{df}
The function $r_{A} : \mathbb{N} \rightarrow \mathbb{N}$ defined by
$$r_{A}(m) = {\rm rk}(A^{m})$$
is called the \emph{rank function} of a matrix $A \in M_{n \times n}(\mathbb{F})$.
\end{df}
\begin{pp}
\label{properties}
For a matrix $A \in M_{n \times n}(\mathbb{F})$ its rank function satisfies the following conditions:
\begin{enumerate}
\item $r_{A}(0) = n$,
\item the function $r_{A}$ is weakly decreasing,
\item $A$ is nilpotent if and only if $r_{A}(n) = 0$,
\item if $r_{A}(m_{0}) = r_{A}(m_{0}+1)$ for some integer $m_{0} \in \mathbb{N}$, then $r_{A}(m_{0}) = r_{A}(m_{0}+i)$ for every $i \in \mathbb{N}$,
\item $r_{U^{-1}AU}(m) = r_{A}(m)$ for every $m \in \mathbb{N}$ and every $U \in GL(n, \mathbb{F})$,
\item if $A = A_{1} \oplus A_{2}$, where $A_{i} \in M_{n_{i} \times n_{i}}(\mathbb{F})$, $i = 1, 2$, and $\oplus$ is the standard direct sum of matrices, then $r_{A}(m) = r_{A_{1}}(m) + r_{A_{2}}(m)$ for all $m \in \mathbb{N}$.
\end{enumerate}
\end{pp}

Rank functions are characterized in the class of all non-negative integer-valued sequences by the following result, see Theorem 2 in \cite{Sk1}.
\begin{tw}
\label{characterization}
A function $r : \mathbb{N} \rightarrow \mathbb{N}$ with $r(0)=n$ is the rank function of a matrix $A \in M_{n \times n}(\mathbb{F})$ if and only if it is weakly decreasing and satisfies the following convexity condition
$$\forall m \in \mathbb{N} : r(m) + r(m+2) \geq 2r(m+1).$$
\end{tw}
Now, we define the main object of our interest.
\begin{df}
\label{df:1}
Let $k, n \in \mathbb{N} \setminus \{0, 1\}$. For fixed functions $f, g : \mathbb{N} \rightarrow \mathbb{N}$ and a nonempty set $S \subseteq \mathbb{N}_{0}$, a \emph{rank function equation} is the equation
\begin{equation}
\label{row1}
f(r_{A_{1}}(m)) + ... + f(r_{A_{k}}(m)) = g(r_{B}(m))
\end{equation}
for all $m \in S$. The indeterminates are matrices $A_{1}, ..., A_{k}, B \in M_{n \times n}(\mathbb{F})$.
\end{df}
For $f = g = {\rm id}_{\mathbb{N}}$, equation (\ref{row1}) reduces to 
\begin{equation}
\label{rowid}
r_{A_{1}}(m) + ... + r_{A_{k}}(m) = r_{B}(m).
\end{equation}
In this note we will consider only non-trivial solutions, which means that  solutions $(A_{1}, ..., A_{k}, B)$ consisting of all nonzero matrices. In \cite{PS} we proved the following theorem.
\begin{tw}
\label{pstw1}
Consider a strictly increasing convex function $f : [0, + \infty) \longrightarrow \mathbb{R}$. Assume that $f(\mathbb{N}) \subseteq \mathbb{N}$ and $f(0)=0$. For nilpotent matrices $A_{1}, ..., A_{k} \in \mathfrak{N}_{n}$ define $r (m) := f (r_{A_{1}} (m)) + ... + f (r_{A_{k}}(m))$.
 Then the following conditions are equivalent:
\begin{enumerate}
\item
there exists a matrix $B \in M_{n \times n} (\mathbb{F})$ such that $(A_1, ..., A_k, B)$ is a solution to equation
\begin{equation}
\label{r1}
f (r_{A_{1}} (m)) + ... + f (r_{A_{k}}(m)) = r_{B}(m)
\end{equation}
with  S = \{1, ..., n\},
\item
$2 r (1) - r (2) \leq n$.
\end{enumerate}
Moreover, if the above conditions are satisfied, then the matrix $B$ is nilpotent and unique up to conjugation.
\end{tw}

\section{Rank Function Equations for non-nilpotent matrices}
Let us recall that for a matrix $A \in M_{n \times n}(\mathbb{F})$ the number $r_{A}(n) \in \mathbb{N}$ is called the \emph{stable rank}. 
We will need the following fact, Proposition 4 in \cite{Sk1}.
\begin{pp}
\label{ppp1}
For nilpotent matrices $A, B \in \mathfrak{N}_{n}(\mathbb{F})$ the following conditions are equivalent:
\begin{itemize}
\item $B = U^{-1}AU$ for a certain $U \in GL(n, \mathbb{F})$,
\item $r_{A}(m) = r_{B}(m)$, $m \in \mathbb{N}$.
\end{itemize}
\end{pp}
The following result characterizes explicitely non-trivial solutions of equation (\ref{rowid}).
\begin{tw}
\label{st1}
Consider equation (\ref{rowid}) with $S = \mathbb{N}_{0}$. Assume that $(A_{1}, ..., A_{k}, B)$ is a solution to (\ref{rowid}), where $A_{j}, B \in M_{n \times n}(\mathbb{F})$, $r_{A_{j}}(n) = q_{j} \in \mathbb{N}_{0}$ for $j = 1, ..., k$ and $n \geq q:=q_{1} + ... + q_{k}$. Then $B$ is similar to $C \oplus D$, where $D \in GL(q, \mathbb{F})$ and $C \in \mathfrak{N}_{n - q}(\mathbb{F})$ is a nilpotent matrix such that its nonzero nilpotent blocks in the Jordan canonical form are conjugate to the direct sum of all nonzero nilpotent blocks contained in the Jordan canonical forms of $A_{1}, ..., A_{k}$.
\end{tw}
\begin{proof}
By [9, Thm. 1] without loss of generality we may assume that $A_{j} = B_{j} \oplus S_{j}$, where $B_{j} \in \mathfrak{N}_{n-q_{j}}(\mathbb{F})$ are nilpotent matrices and $S_{j} \in GL(q_{j}, \mathbb{F})$ for all $j \in \{1, ..., k\}$. Let $\tilde{B_{j}}$ be the direct sum of all nonzero nilpotent blocks that appear in the Jordan canonical form of the matrix $B_{j}$. Then there exist $U_{j} \in GL(n, \mathbb{F})$ such that $U_{j}^{-1}A_{j}U_{j} = \tilde{B_{j}} \oplus S_{j} \oplus O_{n - q_{j} - d_{j}}$ with $d_{j} \in \mathbb{N}$, which depends on the matrices $B_{j}$. Since $(A_{1}, ..., A_{k}, B)$ is a solution to the equation (\ref{rowid}), then for all $m \in \mathbb{N}_{0}$ we have
$$r_{B}(m) = \sum_{j=1}^{k} r_{A_{j}}(m) = \sum_{j=1}^{k} r_{U^{-1}_{j}A_{j}U_{j}}(m) = \sum_{j=1}^{k} r_{\tilde{B}_{j}\oplus S_{j} \oplus O_{n-q_{j}-d_{j}}}(m) = \sum_{j=1}^{k} r_{\tilde{B}_{j}}(m) + q.$$
Obviously $r_{B}(n) = q$, thus $B$ is similar to $C \oplus D$ with $C \in \mathfrak{N}_{n - q}(\mathbb{F})$ and $D \in GL(q, \mathbb{F})$. Then there exists a matrix $V \in GL(n, \mathbb{F})$ such that $V^{-1}BV = \tilde{C} \oplus D \oplus O_{n - q - d_{C}}$, where $d_{C}$ depends on the matrix $C$. Since for $D \in GL(q, \mathbb{F})$ we have $r_{D}(m) = q$ for all $m \in \mathbb{N}$, then
$$\sum_{j=1}^{k} r_{\tilde{B}_{j}}(m) + q = r_{\tilde{B}_{1} \oplus ... \oplus \tilde{B}_{k}}(m) + q = r_{B}(m) = r_{V^{-1} B V}(m) = r_{\tilde{C} \oplus D \oplus O_{n- q -d_{C}}}(m) = r_{\tilde{C}}(m) + q.$$
Now we can focus on the conditions:
$$\left\{
\begin{array}{l} 
\forall\, m \in \mathbb{N}_{0} :\, r_{\tilde{B}_{1} \oplus ... \oplus \tilde{B}_{k}}(m) = r_{\tilde{C}}(m),\\
r_{\tilde{B}_{1} \oplus ... \oplus \tilde{B}_{k}} (n) = r_{\tilde{C}} (n) = 0.
\end{array}
\right.$$
Thus $r_{\tilde{B}_{1} \oplus ... \oplus \tilde{B}_{k}}(0) = r_{\tilde{C}}(0) = d_{C}$ and by Proposition \ref{ppp1} we obtain $\tilde{B_{1}} \oplus ... \oplus \tilde{B_{k}} = W^{-1}\tilde{C}W$ for a certain $W \in GL(d_{C}, \mathbb{F})$, what ends the proof.
\end{proof}
\begin{rem}
Notice that solutions in the above case are not unique (the invertible matrix $D$ can be chosen arbitrarily).
\end{rem}
The next result generalizes Theorem 2.8.
\begin{tw}
\label{st111}
Consider the rank function equation (\ref{r1}) with $S = \{1, ..., n\}$. Let $f : [0, \infty) \longrightarrow \mathbb{R}$ be a strictly increasing convex function such that $f(\mathbb{N}) \subseteq \mathbb{N}$ and $f(0)=0$. Let $A_1, ..., A_k \in M_{n \times n} (\mathbb{F})$ with $r_{A_{j}}(n) = q_{j} \in \mathbb{N}_{0}$ for $j \in \{1, ..., k\}$. Then the following conditions are equivalent:
\begin{enumerate}
\item
there exists a matrix $B \in M_{n \times n}$ such that $(A_1, ..., A_k, B)$ is a solution to (\ref{r1}),
\item
$2 r (1) - r (2) \leq n$, where $r(m) := f(r_{A_{1}}(m)) + ... + f(r_{A_{k}}(m))$ for $m \in \mathbb{N}$.
\end{enumerate}
\end{tw}
\begin{proof}

By [9, Thm. 1] $A_{j}$'s are similar to $\bar{A_{j}} \oplus D_{j}$, where $D_{j} \in GL(q_{j}, \mathbb{F})$ and $\overline{A_{j}} \in \mathfrak{N}_{n - q_{j}}(\mathbb{F})$ are nilpotent matrices for all $j \in \{1, ..., k\}$. Obviously
\begin{center}
$r(n) = f(r_{\bar{A_{1}} \oplus D_{1}}(n)) + ... + f(r_{\bar{A_{k}} \oplus D_{k}}(n)) = f(r_{\bar{A_{1}}}(n) + r_{D_{1}}(n)) + ... + f(r_{\bar{A_{k}}}(n) + r_{ D_{k}}(n)) = f(r_{D_{1}}(n)) + ... + f(r_{D_{k}}(n)) = f(q_{1}) + ... + f(g_{k}).$
\end{center}
In the virtue of Theorem \ref{characterization}, condition $(i)$ holds true iff the function $r_{B} : \mathbb{N} \longrightarrow \mathbb{N}$ defined by
$$r_{B} (m) = \left\{
\begin{array}{ll}
n &\mbox{for $m = 0$,}\\
r (m)&\mbox{for $m \in \{1, ..., n\}$,}\\
\sum _{j=1}^{k} f(q_{j}) &\mbox{for $m > n $}
\end{array}
\right.$$
is weakly decreasing and such that $r_{B} (m) + r_{B} (m + 2) \geq 2 r_{B} (m + 1)$ for all $m \in \mathbb{N}$. By [7, Lemma 3.2] we see that $r_{B}$ is a rank function and thus (i) is satisfied iff $n \geq r_{B} (1)$ and $n + r_{B} (2) \geq 2 r_{B} (1)$. By the monotonicity of $r_{B}$, the last two inequalities hold iff $n - r_{B} (1) \geq r_{B} (1) - r_{B} (2)$, and this is the condition (ii). 
\end{proof}

\section{Some consequences of Gerstenhaber theorem}

For nilpotent matrices $A, B \in \mathfrak{N}_{n}(\mathbb{F})$ we define
$$A \prec B \iff {\rm rk}(A^{m}) \leq {\rm rk}(B^{m}) \,\, {\rm for} \,\, {\rm all} \,\, m \in \mathbb{N} .$$
It can be shown that $\prec$ is a partial order. This order is usually called the \emph{dominance}.

From now on we fix a function $f: [0, \infty) \longrightarrow \mathbb{R}$, which is convex, strictly increasing, $f(0) = 0$ and maps all non-negative integers to non-negative integers. Let us define the following set
\begin{center}
{\rm Sol} = \{ $(A_{1}, ..., A_{k}, B) \in \underbrace{\mathfrak{N} \times ... \times \mathfrak{N}}_{k+1}$ : $(A_{1}, ..., A_{k}, B)$ form a solution to \newline the  equation (\ref{r1}) with fixed $f$ and $S = \mathbb{N}_{0}$ \}.
\end{center}

For $(A_{1}, ..., A_{k}) \in \mathcal{A}$, where $\mathcal{A} \subseteq \underbrace{\mathfrak{N} \times ... \times \mathfrak{N}}_{k}$ is an arbitrary subset, we define the rank matrix 
$${\rm Rk}(A_{1}, ..., A_{k} ) = \left( \begin{array}{ccccc} 
r_{A_{1}}(0) & r_{A_{1}}(1) &  .. & r_{A_{1}}(n-1) & r_{A_{1}}(n) \\
r_{A_{2}}(0) & r_{A_{2}}(1) &  .. & r_{A_{2}}(n-1) & r_{A_{2}}(n) \\
\vdots & \vdots & \vdots &  \vdots & \vdots \\
r_{A_{k}}(0) & r_{A_{k}}(1) &  .. & r_{A_{k}}(n-1) & r_{A_{k}}(n)\\
\end{array} \right).$$

Let us denote by ${\rm Rank(\mathcal{A})} \subseteq M_{k \times (n+1)}(\mathbb{F})$ the set, which consist of all ma\-tri\-ces of the above form. Note that the set ${\rm Rank(\mathcal{A})}$ is always finite. 

For matrices ${\rm Rk}(A_{1}, ...,A_{k}) = [a_{ij}]$ and ${\rm Rk}(A_{1}', ..., A_{k}') =  [a'_{ij}]$, which belong to the set ${\rm Rank(\mathcal{A})}$ we define the relation: \\
\begin{center}
${\rm Rk}(A_{1}, ..., A_{k}) \preceq {\rm Rk}(A_{1}', ..., A_{k}')$ $\iff$ for all $i \in \{1, ..., k\}$ we have $a_{i j} \leq a'_{i j}$ for any $j \in \{1, ..., n+1\}$.
\end{center}
It is easy to see that the relation $\preceq$ is a partial order, which is compatible with the dominance in the sense that for a fixed $i \in \{1, ..., k\}$ we have 
$$a_{ij} \leq a_{ij}' \,\, {\rm for} \,\, {\rm all} \,\, j \in \{1, ..., n+1\} \,\, {\rm iff} \,\, A_{i} \prec A_{i}'.$$

\begin{ex}
Let us consider the rank function equation (\ref{rowid}) with $k=2$.

Let $A_{1}, A_{1}', A_{2}, A_{2}',B, B' \in \mathfrak{N}_{10}(\mathbb{F})$ be nilpotent matrices, such that
\begin{eqnarray}
{\rm jp}(A_{1}) & = & (2, 2, 1, 1, 1, 1, 1, 1) \nonumber , \\
{\rm jp}(A_{2}) & = & (3, 2, 1, 1, 1, 1, 1) \nonumber , \\
{\rm jp}(A_{1}') & = & (4, 1, 1, 1, 1, 1, 1) \nonumber , \\
{\rm jp}(A_{2}') & = & (4, 2, 1, 1, 1, 1) \nonumber , \\
{\rm jp}(B) & = & (3, 2, 2, 2, 1) \nonumber , \\
{\rm jp}(B') & = & (4, 4, 2) \nonumber .
\end{eqnarray}
Then $(A_{1}, A_{2}, B)$ and $(A_{1}', A_{2}', B')$ form solutions to the equation (\ref{rowid}) with $S = \mathbb{N}_{0}$. The rank matrices are the following:
$${\rm Rk}(A_{1}, A_{2}, B) = \left( \begin{array}{cccccccccc} 
10 & 2 &  0 & 0 & 0 & 0 & 0 & 0 & 0 & 0 \\
10 & 3 &  1 & 0 & 0 & 0 & 0 & 0 & 0 & 0 \\
10 & 5 &  1 & 0 & 0 & 0 & 0 & 0 & 0 & 0
\end{array} \right),$$
$${\rm Rk}(A_{1}', A_{2}', B') = \left( \begin{array}{cccccccccc} 
10 & 3 &  2 & 1 & 0 & 0 & 0 & 0 & 0 & 0 \\
10 & 4 &  2 & 1 & 0 & 0 & 0 & 0 & 0 & 0 \\
10 & 7 &  4 & 2 & 0 & 0 & 0 & 0 & 0 & 0
\end{array} \right) .$$
Of course, ${\rm Rk}(A_{1}, A_{2}, B) \preceq {\rm Rk}(A_{1}', A_{2}', B')$.
\end{ex}

It is quite easy to see that the set ${\rm Rank(Sol)}$ may not be totally ordered. Indeed,
\begin{ex}
Consider the rank function equation (\ref{rowid}) with $k=2$. Let $A_{1}, A_{1}', A_{2},$ $ A_{2}',B, B' \in \mathfrak{N}_{8}(\mathbb{F})$ be nilpotent matrices, with
\begin{eqnarray}
{\rm jp}(A_{1}) & = & (2, 1, 1, 1, 1, 1, 1) \nonumber , \\
{\rm jp}(A_{2}) & = & (3, 1, 1, 1, 1, 1) \nonumber , \\
{\rm jp}(A_{1}') & = & (2, 2, 1, 1, 1, 1) \nonumber , \\
{\rm jp}(A_{2}') & = & (2, 2, 1, 1, 1, 1) \nonumber , \\
{\rm jp}(B) & = & (3, 2, 1, 1, 1) \nonumber , \\
{\rm jp}(B') & = & (2, 2, 2, 2) \nonumber .
\end{eqnarray}
The triplets $(A_{1}, A_{2}, B)$ and $(A_{1}', A_{2}', B')$ form solutions to equation (\ref{rowid}) with $S = \mathbb{N}_{0}$, and the rank matrices are the following:
$${\rm Rk}(A_{1}, A_{2}, B) = \left( \begin{array}{cccccccc} 
8 & 1 &  0 & 0 & 0 & 0 & 0 & 0 \\
8 & 2 &  1 & 0 & 0 & 0 & 0 & 0  \\
8 & 3 &  1 & 0 & 0 & 0 & 0 & 0 
\end{array} \right),$$
$${\rm Rk}(A_{1}', A_{2}', B') = \left( \begin{array}{cccccccc} 
8 & 2 &  0 & 0 & 0 & 0 & 0 & 0  \\
8 & 2 &  0 & 0 & 0 & 0 & 0 & 0  \\
8 & 4 &  0 & 0 & 0 & 0 & 0 & 0 
\end{array} \right) .$$
We see that ${\rm Rk}(A_{1}, A_{2}, B)$ and ${\rm Rk}(A_{1}', A_{2}', B')$ are not comparable.
\end{ex}

We denote by ${\rm Rk}_{i}$ the $i$-th row of a matrix ${\rm Rk} \in {\rm Rank(Sol)}$. To a fixed matrix ${\rm Rk} \in {\rm Rank(Sol)}$ we associate a sequence of matrices $(M({\rm Rk}_{1}) , ... , M({\rm Rk}_{k+1}))$, such that $M({\rm Rk}_{i}) \in \mathfrak{N}$ is in the Jordan canonical form, which is defined by the vector ${\rm Rk}_{i}$ in the obvious manner. \\

From now on we will work over the field of complex numbers $\mathbb{C}$. \\

We recall a well-known theorem due to Gerstenhaber \cite{Ger}.
\begin{tw}[Gerstenhaber]
\label{gest}
Let $A, B \in M_{n \times n}(\mathbb{C})$ be nilpotent matrices. Then $A \in \overline{\mathcal{O}(B)}$ if and only if $A \prec B$.
\end{tw}
\begin{lem}
\label{l1}
Let $n \in \mathbb{N}\setminus \{0\}$ and let $A, B \subseteq \mathbb{C}^{n}$ be constructible sets. The following equality holds
$$\overline{A} \times \overline{B} = \overline{A \times B}.$$
\end{lem}
\begin{proof} This is an immediate consequence of the well-known fact that the Zariski and the Euclidean closure of a constructible set coincide.
\end{proof}
The following result generalizes Gerstenhaber's theorem.
\begin{tw}
\label{gg}
Let $A_{1}, ..., A_{k}, B_{1}, ..., B_{k} \in \mathfrak{N}$. Then 
\begin{enumerate}\item $(A_{1}, ..., A_{k}) \in \overline{\mathcal{O}(B_{1}) \times ... \times \mathcal{O}(B_{k})}$ iff ${\rm Rk}(A_{1}, ..., A_{k})$ $\preceq$ ${\rm Rk}(B_{1}, ..., B_{k})$

\item $\mathcal{O}(A_{1}) \times ... \times \mathcal{O}(A_{k}) = \mathcal{O}(B_{1}) \times ... \times \mathcal{O}(B_{k})$ iff ${\rm Rk}(A_{1} , ..., A_{k}) = {\rm Rk}(B_{1}, ..., B_{k})$.
\end{enumerate}
\end{tw}

\begin{proof}
Ad(i). Using Theorem \ref{gest} we have 
\begin{center}
${\rm Rk}(A_{1}, ..., A_{k}) \preceq {\rm Rk}(B_{1}, ..., B_{k})$ if and only if $(A_{1}, ..., A_{k}) \in \overline{\mathcal{O}(B_{1})} \times ... \times \overline{\mathcal{O}(B_{k})}$.
\end{center}
Since ${\rm Rk}(A_{1}, ..., A_{k}) \preceq {\rm Rk}(B_{1}, ..., B_{k})$, thus by Lemma \ref{l1} 
$$(A_{1}, ..., A_{k}) \in \overline{ \mathcal{O}(B_{1}) \times ... \times \mathcal{O}(B_{k})},$$
and the proof of the implication "$\Leftarrow$" in (i) is completed.

Suppose that $(A_{1}, ..., A_{k}) \in \overline{\mathcal{O}(B_{1}) \times ... \times \mathcal{O}(B_{k})}$. By Lemma \ref{l1} we have the equality $$\overline{\mathcal{O}(B_{1})} \times ... \times \overline{\mathcal{O}(B_{k})} = \overline{ \mathcal{O}(B_{1}) \times ... \times \mathcal{O}(B_{k})},$$ which implies that $A_{j} \in \overline{\mathcal{O}(B_{j})}$ for all $j \in \{1, ..., k\}$, and thus $r_{A_{j}} (m) \leq r_{B_{j}} (m)$ for all $m \in \mathbb{N}$. \\
Ad(ii). Implication "$\Rightarrow$" is obvious, other implication is a simple consequence of Proposition \ref{ppp1}.
\end{proof}

Since for constructible sets $A, B \subseteq \mathbb{F}^{n}$, where $\mathbb{F}$ is an arbitrary field, the equality $\overline{A \times B} = \overline{A} \times \overline{B}$ does not hold in general, the proof of Theorem $4.5$ breaks. Nevertheless it would be interesting to know if the result of Theorem $4.5$ remains true. 

\section{Geometry of the set of solutions}

Recall that a set $\mathcal{U}$ is $GL(n, \mathbb{C} )$-invariant if $\mathcal{U}  \supseteq \, \bigcup_{A \in \mathcal{U}}\mathcal{O}(A)$. Thus we see (compare \cite{PS}) that the set ${\rm Sol}$ is $GL(n, \mathbb{C} )$-invariant in the following sense:
$${\rm Sol} \, = \, \bigcup_{(A_{1}, ..., A_{k}, B) \in {\rm Sol}} \mathcal{O}(A_{1}) \times ... \times \mathcal{O}(A_{k}) \times \mathcal{O}(B).$$

Next, wee see that the set ${\rm Sol}$ is a cone, i.e. ${\rm Sol} \neq \emptyset$ and $${\rm Sol} \, \supseteq \, \mathbb{C} {\rm Sol} := \{\lambda(A_{1}, ..., A_{k}, B) : \lambda \in \mathbb{C}, (A_{1}, ..., A_{k}, B) \in {\rm Sol}\}.$$

Recall also (see for example \cite{Sk3}) that for a matrix $A \in \mathfrak{N}$ the dimension of $\overline{\mathcal{O}(A)}$ can be computed by the following formula
$$\dim(\overline{\mathcal{O}(A)}) = n^{2} - \sum_{j=0}^{\infty}(r_{A}(j) - r_{A}(j+1))^2.$$

We denote by ${\rm Sol}_{\rm id}$ the set of all nilpotent solutions to the rank function equation (\ref{rowid}) with $S = \{1, ..., n\}$.

\begin{ex}
It is easy to see that the set $\overline{{\rm Sol}}_{\rm id}$ with $n = 2k$ is irreducible and has the form
$$\overline{{\rm Sol}}_{{\rm id}} = \overline{ \underbrace{\mathcal{O}(N_{2} \oplus O_{2k-2}) \times ... \times \mathcal{O}(N_{2} \oplus O_{2k-2})}_{k} \times \mathcal{O}(\underbrace{N_{2} \oplus ... \oplus N_{2}}_{k})}.$$
The dimension of this set is equal $\dim \, \overline{{\rm Sol}}_{{\rm id}} = 6k^2 - 2k$ (we omit some dull computations).
\end{ex}

\begin{ex} Let us consider the set $\overline{{\rm Sol}}_{\rm id}$ with $n = 2k+1$. It is quite easy to see this set is reducible and has the form

$$\overline{{\rm Sol}}_{\rm id} = \bigcup_{i=1}^{k} \overline{\mathcal{O}(N_{2} \oplus O_{2k-1}) \times ... \times \mathcal{O}(N_{3} \oplus O_{2k-2})^{[i]} \times ... \times \mathcal{O}(N_{2} \oplus O_{2k-1}) \times \mathcal{O}(N_{3} \oplus \underbrace{N_{2} \oplus ... \oplus N_{2}}_{k-1})},
$$
where $[i]$ denotes the $i$-th position, counted from the left-hand side, on which the conjugacy class $\mathcal{O}(N_{3} \oplus O_{2k-2})$ appears. We see that all irreducible components have the same dimension, and $\dim \overline{{\rm Sol}}_{\rm id} = 6k^2 +8k -2$.

\end{ex}

One of the most important consequences of Theorem \ref{gest} is that if $\mathcal{U} \subseteq M_{n \times n}(\mathbb{F})$ is a $GL(n, \mathbb{F})$-invariant set of nilpotent matrices over an algebraic closed field of characteristic zero, then there is a bijective correspondence between irreducible components of $\mathcal{U}$ and maximal elements of the set $R(\mathcal{U}) := \{r_{A} \, : A \in \mathcal{U}\}$ in the sense of order $\prec$. The same can be shown in our case. 

We will follow \cite{Sk2}. We denote by $\mathcal{Z}({\rm Sol} )$ the set of all maximal elements of the set ${\rm Rank(Sol)}$ in the sense of order $\preceq$ and by $\mathcal{I}({\rm Sol})$ the familly of all irreducible components of $\overline{{\rm Sol} }$.

It is well known that conjugacy classes of matrices are irreducible constructible sets, and thus the cartesian products of them are irreducible. By Theorem \ref{gg} and the fact that for each matrix ${\rm Rk} \in {\rm Rank(Sol)}$ there exists an ${\rm Rk}_{0} \in \mathcal{Z}({\rm Sol})$ such that ${\rm Rk} \preceq {\rm Rk}_{0}$ we obtain the following result.

\begin{tw} 
The maps
$$\mathcal{Z}({\rm Sol} ) \ni {\rm Rk} \mapsto \overline{\mathcal{O}(M({\rm Rk}_{1}))\times ... \times \mathcal{O}(M({\rm Rk}_{k+1}))} \in \mathcal{I}({\rm Sol} ),$$
$$\mathcal{I}({\rm Sol} ) \ni W \mapsto \max  {\rm Rank}(W) \in \mathcal{Z}({\rm Sol} )$$
are well defined mutually inverse bijections.
\end{tw} 

\begin{col}
\label{11} For the set ${\rm Sol} \subset \underbrace{\mathfrak{N} \times ... \times \mathfrak{N}}_{k+1}$ the following conditions are equivalent:
\begin{itemize}
\item $\overline{{\rm Sol}}$ is an irreducible set,
\item there is a greatest element in ${\rm Rank(Sol)}$ with respect to $\preceq$.
\end{itemize}
\end{col}

\section{The linear capacity}
At the end of the note we would like to formulate some remarks about the linear capacity of solution sets, which will be denoted by $\Lambda(\overline{{\rm Sol}})$. 

\begin{df}
Let $\mathcal{E} \subset M_{n \times n}(\mathbb{F})$, where $\mathbb{F}$ is an algebraically closed field of characteristic zero. The linear capacity $\Lambda(\mathcal{E}) \in \mathbb{N} \cup \{-\infty\}$ is defined by the formula
$$\Lambda(\mathcal{E}) = \max \{ \dim \, \mathcal{L} \, : \, \mathcal{L} \,\, {\rm is} \,\, {\rm linear} \,\, {\rm subspace} \,\, {\rm of} \,\, M_{n \times n}(\mathbb{F}), \,\, L \subseteq \overline{\mathcal{E}} \}.$$
\end{df}
\begin{pp} Let $\mathcal{E}_{1}, \mathcal{E}_{2} \subseteq M_{n \times n}(\mathbb{F})$, with $\mathbb{F}$ as above, be such that $\overline{\mathcal{E}_{1}}, \overline{\mathcal{E}_{2}}$ are irreducible algebraic cones. Then
\begin{enumerate}
\item $\Lambda(\overline{\mathcal{E}_{1}}) = \Lambda(\mathcal{E}_{1})$.
\item $\Lambda(\mathcal{E}_{1} \times \mathcal{E}_{2}) = \Lambda(\mathcal{E}_{1}) + \Lambda(\mathcal{E}_{2})$
\end{enumerate}
\end{pp}

\begin{tw}[Gerstenhaber, Chavey - Brualdi] Let $\mathbb{F}$ be an algebraically closed field of characteristic zero.
If $B \in \mathfrak{N}_{n}(\mathbb{F})$ is a nilpotent matrix, then 
$$\Lambda(\mathcal{O}(B)) = \frac{1}{2} \dim(\overline{\mathcal{O}(B)}).$$
\end{tw}
\begin{ex}
We compute the linear capacity of $\overline{{\rm Sol}}_{\rm id} \subseteq \underbrace{M_{n \times n}(\mathbb{C}) \times ... \times M_{n \times n}(\mathbb{C})}_\text{k+1}$ with $n=2k$, namely
\begin{center}
$\Lambda(\overline{{\rm Sol}}_{\rm id}) = \Lambda({\rm Sol}_{\rm id}) = \Lambda(\mathcal{O}(A_{1}) \times ... \times \mathcal{O}(A_{k}) \times \mathcal{O}(B)) = \Lambda(\mathcal{O}(A_{1})) + ... + \Lambda(\mathcal{O}(A_{k})) + \Lambda(\mathcal{O}(B)) = (\dim \overline{ \mathcal{O}(A_{1})} + ... +  \dim \overline{ \mathcal{O}(A_{k})} + \dim \overline{ \mathcal{O}(B)})/2 = 3k^2 - k.$
\end{center}
\end{ex}
All these facts and reducibility of the set $\overline{{\rm Sol}}$ lead to the following modification of the above definition of linear capacity
\begin{center}
$\Lambda(\overline{{\rm Sol}}) = \max \{\dim \mathcal{L} \, : \, \mathcal{L} \,\, {\rm is} \,\, {\rm linear} \,\, {\rm subspace} \,\, {\rm of} \,\, \underbrace{M_{n \times n}(\mathbb{C}) \times ... \times M_{n \times n}(\mathbb{C})}_\text{k+1}, \newline \mathcal{L} \subseteq \overline{\mathcal{O}(A_{1}^{i}) \times ... \times \mathcal{O}(A_{k}^{i}) \times \mathcal{O}(B^{i})}, i \in I \},$
\end{center}
where $I \subset \mathbb{N}_{0}$ is such that $\# \, I$ is equal to the number of irreducible components of $\overline{{\rm Sol}}$. Of course, if $\overline{{\rm Sol}}$ is irreducible, then by Corollary \ref{11} the above maximum of dimension can be attained. Moreover, it is quite easy to find an upper bound of $\Lambda(\overline{{\rm Sol}})$.
\begin{pp}
\label{ppp} For the set ${\overline{\rm Sol}}$ there exist matrices $C_{1}, ..., C_{k}, D \in \mathfrak{N}$ such that 
$$\overline{{\rm Sol}} = \, \bigcup_{(A_{1}, ..., A_{k}, B) \in {\rm Sol}} \overline{ \mathcal{O}(A_{1}) \times ... \times \mathcal{O}(A_{k}) \times \mathcal{O}(B)} \subseteq \overline{\mathcal{O}(C_{1}) \times ... \times \mathcal{O}(C_{k}) \times \mathcal{O}(D)} \subsetneq \mathfrak{N}^{k+1}.$$
\end{pp}
\begin{proof}
If the set ${\overline{\rm Sol}}$ is irreducible, there is nothing to prove. Suppose that the set ${\overline{\rm Sol}}$ is reducible. It is enough to prove that for a set ${\rm Rank(Sol)}$ we can find an element ${\rm Rk}_{0} \in {\rm Rank}(\mathfrak{N})$ such that ${\rm Rk} \preceq {\rm Rk}_{0}$ for every ${\rm Rk} \in {\rm Rank(Sol)}$. We construct the matrix ${\rm Rk}_{0} \in {\rm Rank}(\mathfrak{N})$ using the below procedure:
\begin{enumerate}
\item For a fixed $n \in \mathbb{N}_{0}$ - this number depends on the set $\overline{{\rm Sol}}$ - by Theorem \ref{gest} we can construct the diagram of all posible Jordan partitions ordered by the domination $\prec$ (the precise construction can be found in [6, Example 2.12]).
\item We consider the set of all nilpotent matrices from the first coordinate of the set $\overline{{\rm Sol}}$ and we denote it by ${\rm Cr}_{1}$. For all matrices in ${\rm Cr}_{1}$ we find their Jordan partitions. Notice that the set of all such Jordan partitions is finite.
\item Using the diagram from step $1$ we find a matrix $C_{1} \in \mathfrak{N}$, which dominates all matrices from the set ${\rm Cr}_{1}$ - of course such a matrix always exists. Moreover, the matrix $C_{1}$ can be chosen in a such way that its Jordan partition is different from $(n)$ (see [7, Prop. ~4.2~]).
\item We continue this procedure for another coordinates. In the consequence we find matrices $C_{1}, ..., C_{k}, D \in \mathfrak{N}$, such that 
$${\rm Rk} \preceq {\rm Rk}(C_{1}, ..., C_{k}, D) := {\rm Rk}_{0}, \,\,\, \forall \, {\rm Rk} \, \in \, {\rm Rank(Sol)}.$$
\end{enumerate} 
Thus Theorem \ref{gg} gives us the desired inclusion.

The strict containment
$\overline{\mathcal{O}(C_{1}) \times ... \times \mathcal{O}(C_{k}) \times \mathcal{O}(D)} \subsetneq \mathfrak{N}^{k+1}$ is the consequence of step (iii) in the above procedure.
\end{proof}

Using the above proposition we find quite obvious relation
$$\Lambda(\overline{{\rm Sol}}) \leq \frac{1}{2}\bigg(\dim(\overline{\mathcal{O}(C_{1})}) + ... + \dim(\overline{\mathcal{O}(C_{k})}) + \dim(\overline{\mathcal{O}(D)})\bigg).$$

\section{Further Problems}
In two consecutive notes we examine some rank function equations in the case of singular matrices. However our all attention was focused on the case with $f$ satisfying some natural properties and $g$ as the identity function and thus we are curious what can happen if $g(m) \neq m$. These lead us to the following problem.
\begin{qu} Describe all possible nilpotent solutions (i.e. a solution consisting of nilpotent matrices) to the equation (\ref{row1}) with a fixed function $f$ satisfing the same properties as in the hypothesis of Theorem \ref{pstw1}, a fixed function $g(m) \neq m$ and $S = \mathbb{N}_{0}$.
\end{qu}
The easiest example of such equation is 
\begin{equation}
\label{pita}
[r_{A}(m)]^2 + [r_{B}(m)]^2 = [r_{C}(m)]^2
\end{equation}
with $A,B,C \in \mathfrak{N}$. This equation is somehow connected with the famous Pythagorean equation and, what is not so surprising, we can construct some nilpotent solutions to (\ref{pita}) using Pythagorean triples -- see Example 2.7 $\&$ 2.8 in \cite{PS}. It is natural to formulate the following question.
\begin{qu}
Are there other solutions to the equation (\ref{pita}) than mentioned above ?
\end{qu}

Of course, the same questions can be formulated for singular non-nilpotent matrices.

In Section $4$, we investigate a certain generalization of Theorem \ref{gest} in the special case $\mathbb{F} = \mathbb{C}$. Since the method of proving this theorem breaks in the case of other fields, it is really interesting whether this theorem is still valid for an arbitrarily chosen algebraically closed field $\mathbb{F}$.

If $g(m) \neq m$, then some of methods presented in these notes break and thus we need new tools and ideas, even for the equation (\ref{pita}), because there is no natural rank function, which allows to mimic our argumentations.
\section*{Acknowledgment}
The author would like to express his gratitude to prof. Tomasz Szemberg and prof. Kamil Rusek for helpful suggestions.

Piotr Pokora, (Institute of Mathematics) Pedagogical University of Cracow, Podchor\c{a}\.zych 2, 30-084 Krak\'ow (Poland)
\\ e-mail: piotrpkr@gmail.com

\end{document}